\documentclass[12pt]{article}
\usepackage[T1]{fontenc}
\usepackage{lmodern}
\usepackage{amsmath}
\usepackage{amssymb}
\usepackage{latexsym}
\usepackage{amsthm}
\usepackage{mathrsfs}
\usepackage{enumitem}
\usepackage[colorlinks,
            linkcolor=red,
            anchorcolor=green,
            citecolor=blue
            ]{hyperref}

\newtheorem{thm}{Theorem}[section]
\newtheorem{lem}[thm]{Lemma}

\newtheorem{cor}[thm]{Corollary}
\newtheorem{conj}[thm]{Conjecture}

\def\ZZ{{\mathbb Z}}

\def\des{{\rm des}}

\def\maj{{\rm maj}}
\def\maj{{\rm maj}}
\def\fmaj{{\rm fmaj}}

\newcommand{\sep}{\preceq}

\begin{document}

\begin{center}
{\large \bf Interlacing Polynomials and the Veronese Construction for Rational Formal Power Series}
\end{center}

\begin{center}
Philip B. Zhang\\[6pt]

College of Mathematical Science \\
Tianjin Normal University, Tianjin  300387, P. R. China\\[8pt]

Email: {\tt zhangbiaonk@163.com}
\end{center}

\noindent\textbf{Abstract.}
Fixing a positive integer $r$ and $0 \le k \le r-1$, define $f^{\langle r,k \rangle}$ for every formal power series $f$ as $
 f(x) = f^{\langle r,0 \rangle}(x^r)+xf^{\langle r,1 \rangle}(x^r)+ \cdots +x^{r-1}f^{\langle r,r-1 \rangle}(x^r).$
Jochemko recently showed that the polynomial $U^{n}_{r,k}\, h(x) := \left( (1+x+\cdots+x^{r-1})^{n} h(x) \right)^{\langle r,k \rangle}$
has only nonpositive zeros for any  $r \ge  \deg h(x) -k$ and any positive integer $n$. As a consequence, Jochemko confirmed a conjecture of Beck and Stapledon on the Ehrhart polynomial $h(x)$ of a lattice polytope of dimension $n$, which states that
$U^{n}_{r,0}\,h(x)$ has only negative, real zeros whenever $r\ge n$. In this paper, we provide an alternative approach to Beck and Stapledon's conjecture by proving the following general result:
if the polynomial sequence $\left( h^{\langle r,r-i \rangle}(x)\right)_{1\le i \le r}$ is interlacing, so is  $\left( U^{n}_{r,r-i}\, h(x) \right)_{1\le i \le r}$. Our result has many other interesting applications. In particular, this enables us to give a new proof of Savage and Visontai's result on the interlacing property of some refinements of the descent generating functions for colored permutations.
Besides, we derive a Carlitz identity for refined colored permutations.

\noindent \emph{AMS Classification 2010:} Primary 05A15, 13A02, 26C10; Secondary 52B20, 52B45.


\section{Introduction}

Suppose that $\{a_i\}_{i\geq 0}$ is a sequence of real numbers that eventually agrees with a polynomial, namely there exist $i_0\geq 0$ and a polynomial $g$ in $i$ of degree $n-1\geq 0$ such that $a_i=g(i)$ for any $i\geq i_0$. It is well known that the generating series $$\sum_{i=0}^{\infty} a_i\, x^{i}$$
converges to a rational function of the form
$$\frac{h(x)}{(1-x)^{n}},$$
where $h(x)$ is a polynomial with $h(1)\neq 0$, see \cite{Stanley2012Enumerative}.
Fix a positive integer $r$. Clearly, for any $0 \le k \le r-1$, the generating series
$$\sum_{i=0}^{\infty} a_{ri+k} \, x^{i}$$
also converges to a rational function of the same form
$$\frac{\bar{h}(x)}{(1-x)^{n}},$$
where $\bar{h}(x)$ is a polynomial with $\bar{h}(1)\neq 0$. As will be shown later, the polynomial $\bar{h}(x)$ is uniquely determined by
$n,r,k$ and $h(x)$. To make this point clear, we denote
$\bar{h}(x)$ by $U^{n}_{r,k}\, h(x)$.
Such a construction originates from the Veronese subalgebra of graded algebra \cite{Beck2010log}.
Given a formal power series $f(x)$, there exist uniquely determined series $f^{\langle r,0 \rangle}(x),\,f^{\langle r,1 \rangle}(x),\,\ldots,\,f^{\langle r,r-1 \rangle}(x)$ such that
\begin{align*}
 f(x) = f^{\langle r,0 \rangle}(x^r)+xf^{\langle r,1 \rangle}(x^r)+ \cdots +x^{r-1}f^{\langle r,r-1 \rangle}(x^r).
\end{align*}
It was shown in \cite[Lemma 3.1]{Jochemkotoappearreal} that
\begin{align}\label{eq:main}
 U_{r,k}^{n}\, h(x)  =  \left((1+x+\cdots+x^{r-1})^{n} h(x)\right)^{\langle r,k \rangle}.
\end{align}

Rational formal power series of the above form $\frac{U_{r,k}^{n}\, h(x)}{(1-x)^{n}}$ have appeared widely in combinatorics \cite{Athanasiadis2014Edgewise,  Beck2015Computing, Brun2005Subdivisions, JochemkotoappearCombinatorial} and commutative algebra \cite{Bruns1993Cohen, Eisenbud1994Initial, Kubitzke2012Enumerative}. For example, given an $(n-1)$-dimensional lattice polytope $P$, the Ehrhart series of $P$ is
$$\sum_{i\geq 0}E_P(i)x^i=\frac{h^*(P)(x)}{(1-x)^n},$$
where $E_P(i)$ counts the number of lattice points in $iP$ (the $i$-th dilate of $P$), and $h^*(P)(x)$ is the Ehrhart $h^*$-polynomial of $P$ with degree less than $n$. For the $r$-th dilate of $P$, it can be shown that
$$h^*(rP)(x)=U_{r,0}^{n}\, h^*(P)(x).$$
Recently, there is an arising interest in the study of the real-rootedness of $U_{r,k}^{n}\, h(x)$ when the numerator polynomial $h(x)$ has only nonnegative coefficients.  Brenti and Welker \cite{Brenti2009Veronese} proved that, for a sufficiently large integer $r$, the polynomial $U^{n}_{r,0}\,h(x)$ has only distinct, negative, real zeros.
Beck and Stapledon \cite{Beck2010log} showed that, given a positive integer $n$, there exists an $R > 0$  such that
if $h(x)$ is a polynomial of degree less than $n$ with nonnegative integer coefficients and $h(0)=1$, then
the polynomial $U^{n}_{r,0}\,h(x)$ has only distinct, negative, real zeros for every integer $r > R$. In the case of Ehrhart polynomials,  Beck and Stapledon proposed the following conjecture.

\begin{conj}[{\cite[Conjecture 5.1]{Beck2010log}}]\label{conj-Beck}
Let $h(x)$ be the Ehrhart polynomial of a lattice polytope of dimension $n$. Then $U^{n}_{r,0}\,h(x)$ has only real zeros whenever $r\ge n$.
\end{conj}

The original conjecture of Beck and Stapledon requires that $U^{n}_{r,0}\,h(x)$ should have distinct, negative, real zeros.
Here we only focus on its real-rootedness.
This conjecture was recently proved by Jochemko \cite{Jochemkotoappearreal}, who obtained the following more general result.

\begin{thm}[{\cite[Theorem 1.1]{Jochemkotoappearreal}}]\label{thm:J}
Given positive integers $r,n$ and $0\leq k\leq r-1$, if $h(x)$ is a polynomial of degree $\leq r+k$ with nonnegative coefficients, then the polynomial $U_{r,k}^{n}\,h(x)$ has only real zeros.
\end{thm}

Note that taking $k=0$ in Theorem \ref{thm:J} we obtain Conjecture \ref{conj-Beck}. Jochemko's proof of the above theorem involves the notion of interlacing, which has been  used to prove the real-rootedness of several polynomials arising in  combinatorics (\cite{Jochemkotoappearreal, Leander2016Compatible, Savage2015s, Yang2014Mutual, Zhang2015real}).
Suppose that $f(x)$ and $g(x)$ are  two real-rooted polynomials with positive leading coefficients. We say that {$g(x)$ interlaces $f(x)$}, denoted $g(x)\sep f(x)$, if
$\deg f(x)=\deg g(x)$ or $\deg f(x)=\deg g(x)+1$, and
the zeros of $g(x)$ and $f(x)$
 alternate in the following way:
 \begin{align*}
\cdots\le v_2\le u_2\le v_1\le u_1,
\end{align*}
where $u_i$ and $v_j$ are the zeros of $f(x)$ and $g(x)$, respectively.
For convention, we let $a\sep bx+c$ for any nonnegative $a,b,c$, and let $0 \sep f$ and $f \sep 0$ for any real-rooted polynomial $f$.
We say that a sequence of real polynomials $(f_{1}(x),\dots,f_{m}(x))$
with positive leading coefficients is {interlacing} if $f_{i}(x) \sep f_{j}(x)$ for all $1\le i<j\le m$.
For more information on interlacing polynomials, we refer the reader to Br{\"a}nd{\'e}n's survey \cite{Braenden2015Unimodality}.

The proof of Theorem \ref{thm:J} relies on the following result due to Fisk \cite{Fisk2006Polynomials}.

\begin{lem}[{\cite[Example 3.76]{Fisk2006Polynomials}}]
\label{lem:a}
Let $I$ denote the constant function $h(x)=1$ and let
$U_{r,k}^{n}\,I$
be defined as in \eqref{eq:main} for positive $n,r$ and nonnegative $0\leq k\leq r-1$.
Then the sequence $\left(U_{r,r-1}^{n}\,I,\ldots, U_{r,1}^{n}\,I, U_{r,0}^{n}\,I\right)$
is interlacing.
\end{lem}

The main result of this paper is as follows.

\begin{thm}\label{thm:main-hnb}
Suppose that $n,r$ are positive integers, and $h(x)$ is a polynomial with nonnegative coefficients. If the polynomial sequence  $\left( h^{\langle r,r-1 \rangle}(x),\ldots, h^{\langle r,1 \rangle}(x), h^{\langle r, 0 \rangle}(x) \right)$ is interlacing, then so is the polynomial sequence  $\left(U_{r,r-1}^{n}\,h(x),\ldots, U_{r,1}^{n}\,h(x), U_{r,0}^{n}\,h(x)\right)$.
\end{thm}

The above theorem has many immediate consequences. For example, the $k=0$ case of Theorem \ref{thm:J} can be derived from Theorem \ref{thm:main-hnb}. Theorem
\ref{thm:main-hnb} also generalizes Lemma \ref{lem:a} since, when $h(x)$ is the constant function of value $1$, the polynomial sequence $\left( h^{\langle r,r-1 \rangle}(x),\ldots, h^{\langle r,1 \rangle}(x), h^{\langle r, 0 \rangle}(x) \right)$ is clearly interlacing.
As shown by Jochemko, Theorem \ref{thm:J} is a consequence of Lemma \ref{lem:a}. For this reason, Theorem
\ref{thm:main-hnb} can also be considered as a generalization of Theorem \ref{thm:J}.


The rest of the paper is organized as follows. We first give a proof of Theorem \ref{thm:main-hnb} in Section \ref{sect-interlacing}. Then some of its applications are presented in Section \ref{sect-applications}, including the derivation of the $k=0$ case of Theorem \ref{thm:J} from Theorem \ref{thm:main-hnb}. One of the most interesting applications is to show that the polynomial sequence  $\left(U_{r,r-1}^{n}\,h(x),\ldots, U_{r,1}^{n}\,h(x), U_{r,0}^{n}\,h(x)\right)$ is interlacing provided that $h(x)$ is real-rooted. Based on this criterion, we further give an alternative proof of the interlacing property of some refinements of the descent generating functions for colored permutations, which is essentially due to Savage and Visontai \cite[Section~3.3]{Savage2015s}. Two key identities are used
in our proof of Savage and Visontai's result, which turn out to be new and will be proved in Section \ref{sect-Carlitz}. To prove these two identities, we establish a Carlitz identity for refined colored permutations in Section \ref{sect-Carlitz}.

\section{Proof of Theorem \ref{thm:main-hnb}}
\label{sect-interlacing}

The objective of this section is to prove Theorem \ref{thm:main-hnb}.
Our proof relies on the theory of interlacing polynomials.

Before giving our proof, let us first recall two results due to Fisk \cite{Fisk2006Polynomials} which will be used later. The first one is stated as follows.

\begin{lem}[{\cite[Example 3.76]{Fisk2006Polynomials}}]\label{lem:fg}
Let $r\ge2$ be an integer. Given a polynomial sequence with nonnegative coefficients   $\left(f_{r-1}(x), \dots, f_{1}(x), f_{0}(x)\right)$, define another sequence
$\left(g_{r-1}(x), \dots, g_{1}(x), g_{0}(x)\right)$ by
\begin{align}
 \left(1+x+\dots+x^{r-1}\right) \left( f_{0}(x^r)+xf_1(x^r)+\dots+x^{r-1}f_{r-1}(x^r) \right) \notag \\
 =g_{0}(x^r)+xg_{1}(x^r)+\dots+x^{r-1}g_{r-1}(x^r). \label{eq:fg}
\end{align}
If  $\left(f_{r-1}(x), \dots, f_{1}(x), f_{0}(x)\right)$  is interlacing, then so is $\left(g_{r-1}(x), \dots, g_{1}(x), g_{0}(x)\right)$.
\end{lem}

Note that successive applications of the above result to
$$\left(f_{r-1}(x), \dots, f_{1}(x), f_{0}(x)\right)=(0,\dots,0,1)$$
and the intermediate polynomial sequences
lead to Lemma \ref{lem:a}.

The second result we shall use is as follows.

\begin{lem}[{\cite[Lemma 3.16]{Fisk2006Polynomials}}]\label{lem:fg_r}
 If both $\left(f_{1}(x), f_{2}(x), \ldots, f_{n}(x)\right)$ and $\left(g_{1}(x), g_{2}(x), \ldots, g_{n}(x)\right)$ are interlacing, then the polynomial
 \begin{align*}
  \sum_{i=1}^{n} f_{i} g_{n+1-i} = f_1 g_n + f_2 g_{n-1} + \cdots + f_n g_1
 \end{align*}
has only real zeros.
\end{lem}

We proceed to prove Theorem \ref{thm:main-hnb}.
\begin{proof}[Proof of Theorem \ref{thm:main-hnb}]
We first show the real-rootedness of $U_{r,i}^{n} h(x)$.
By \eqref{eq:main} the polynomial $U_{r,i}^{n}\, h(x)$ can be expressed in terms of $h^{\langle r,j \rangle}(x)$ and $U_{r,j}^{n} I$ as follows:
\begin{align}
 U_{r,i}^{n}\, h(x) = & \ \left( h^{\langle r,i \rangle}(x) \, U_{r,0}^{n} I    + \cdots + h^{\langle r,0\rangle}(x) \, U_{r,i}^{n} I  \right) \notag \\
 &\ + x \left( h^{\langle r,r-1\rangle}(x) \, U_{r,i+1}^{n} I    + \cdots + h^{\langle r,i+1\rangle}(x) \, U_{r,r-1}^{n} I  \right). \label{eq:uri}
\end{align}
By Lemma \ref{lem:a},
we know that $$\left(U_{r,r-1}^{n}\,I,\ldots, U_{r,1}^{n}\,I, U_{r,0}^{n}\,I\right)$$ is interlacing. Since $$( h^{\langle r,r-1 \rangle}(x), \ldots, h^{\langle r,1 \rangle}(x), h^{\langle r, 0 \rangle}(x) )$$ is interlacing and all the zeros of these polynomials are negative, it follows that
$$\left( h^{\langle r,i \rangle}(x), \ldots, h^{\langle r,0 \rangle}(x), x h^{\langle r,r-1 \rangle}(x), \ldots, x h^{\langle r,i+1 \rangle}(x) \right)$$ is also interlacing.
 Hence, the real-rootedness of $U^{r,i}_{n}\, h(x)$ follows from Lemma \ref{lem:fg_r}.

We next prove the sequence $$\left(U_{r,r-1}^{n}\,h(x),\ldots, U_{r,1}^{n}\,h(x), U_{r,0}^{n}\,h(x)\right)$$ is interlacing. This can be done by induction on $n$.
For the initial case $n=0$, the polynomial sequence $( h^{\langle r,r-1 \rangle}(x), \ldots, h^{\langle r,1 \rangle}(x), h^{\langle r, 0 \rangle}(x) )$ is interlacing. Assume the argument is true for $n$.
Since
\begin{align*}
 \left(1+x+\dots+x^{r-1}\right) \left( U_{r,0}^{n}\,h(x^r)+ x U_{r,1}^{n}\,h(x^r)+\dots+x^{r-1} U_{r,r-1}^{n}\,h(x^r) \right)  \\
 =\left( U_{r,0}^{n+1}\,h(x^r)+ x U_{r,1}^{n+1}\,h(x^r)+\dots+x^{r-1} U_{r,r-1}^{n+1}\,h(x^r) \right),
\end{align*}
 the interlacing property of $\left(U_{r,r-1}^{n+1}\,h(x),\ldots, U_{r,1}^{n+1}\,h(x), U_{r,0}^{n+1}\,h(x)\right)$ immediately follows from Lemma \ref{lem:fg}. This completes the proof.
\end{proof}

\section{Applications}\label{sect-applications}

In this section we shall present some applications of Theorem \ref{thm:main-hnb}. As will be shown,
this leads to useful criteria to determine the interlacing property of $\left( U_{r,r-1}^{n}\, h(x), \ldots, U_{r,0}^{n}\, h(x) \right)$.

First, let us show how to apply Theorem \ref{thm:main-hnb}
to prove Beck and Stapledon's conjecture. Since Conjecture \ref{conj-Beck} is implied by the $k=0$ case of Theorem \ref{thm:J}, we shall directly show how to derive the latter from Theorem \ref{thm:main-hnb}. For $k=0$, the hypothesis of Theorem \ref{thm:J} tells us that $\deg h(x)\leq r$. Thus the polynomial $h^{\langle r,i\rangle}(x)$ is just a constant for each $1\le i \le r-1$, and $h^{\langle r,0\rangle}(x)$ is either a constant or a linear polynomial.
It is clear that $\left( h^{\langle r,r-1 \rangle}(x),\ldots, h^{\langle r,1 \rangle}(x), h^{\langle r, 0 \rangle}(x) \right)$ is interlacing. Applying Theorem \ref{thm:main-hnb}, we get
the following result.

\begin{cor}\label{thm:deg_h<=r}
Given a positive integer $r$, if $h(x)$ is a polynomial of degree $\leq r$ with  nonnegative coefficients, then the sequence $\left( U_{r,r-1}^{n}\, h(x), \ldots, U_{r,0}^{n}\, h(x) \right)$ is interlacing.
\end{cor}

The above corollary implies that each $U_{r,j}^{n}\, h(x)$ is real-rooted for $0\leq j\leq r-1$. Thus we obtain another proof of the $k=0$ case of Theorem \ref{thm:J}, and hence that of Conjecture \ref{conj-Beck}.

The second application of Theorem \ref{thm:main-hnb} is concerned with a criterion for determining the interlacing property of $\left( U_{r,r-1}^{n}\, h(x), \ldots, U_{r,0}^{n}\, h(x) \right)$ by imposing restrictions on the coefficients of $h(x)$.

\begin{cor}\label{thm:deg_h<2r}
Fixing a positive integer $r$, let $h(x)$ be a polynomial with nonnegative coefficients of degree less than $2r$, say
$$h(x) = h_0 + h_1 x +h_2 x^2 + \cdots + h_{2r-1}x^{2r-1}.$$
If its coefficients satisfy the following inequalities
\[
h_{i}h_{r+j}\le h_{j}h_{r+i}, \quad \mbox{for any } 0\leq i<j\le r-1,
\]
then the polynomial sequence $\left(U_{r,r-1}^{n}\,h(x),\ldots, U_{r,1}^{n}\,h(x), U_{r,0}^{n}\,h(x)\right)$ is interlacing.
\end{cor}

\begin{proof}
By Theorem \ref{thm:main-hnb}, it suffices to show
\begin{align}\label{eq:ineq}
h^{\langle r,j \rangle}(x) = h_j+h_{r+j}x \sep h_i+h_{r+i}x=h^{\langle r,i \rangle}(x)
\end{align}
for any $0\leq i<j\le r-1$. There are three cases to consider.
\begin{itemize}
\item[(i)] If $h_{r+j}h_{r+i}\neq 0$, we only need to prove that $$-\frac{h_j}{h_{r+j}}\le -\frac{h_i}{h_{r+i}},$$
which is equivalent to the hypothesis $h_{i}h_{r+j}\le h_{j}h_{r+i}$.

\item[(ii)] If $h_{r+j}=0$, then $h^{\langle r,j \rangle}(x)$ is a constant. By convention, \eqref{eq:ineq} is true.

\item[(iii)] If $h_{r+i}=0$ and $h_{r+j}\neq 0$, then from the hypothesis $h_{i}h_{r+j}\le h_{j}h_{r+i}$ it follows that $h_{i}h_{r+j}=0$. Since $h_{r+j}\neq 0$, we must have $h_{i}=0$ and hence $h^{\langle r,i \rangle}(x)=0$.  By convention, \eqref{eq:ineq} is true.
\end{itemize}
Combining all the above cases, we prove that $\left( h^{\langle r,r-1 \rangle}(x), \ldots, h^{\langle r,1 \rangle}(x), h^{\langle r, 0 \rangle}(x) \right)$ is interlacing. Then applying Theorem \ref{thm:main-hnb} we obtain the desired result.
\end{proof}

Corollary \ref{thm:deg_h<2r} also enables us to give another criterion to determine the interlacing property of $\left(U_{r,r-1}^{n}\,h(x),\ldots, U_{r,1}^{n}\,h(x), U_{r,0}^{n}\,h(x)\right)$ when $\deg h(x)<2r$. This new criterion uses the log-concavity of the coefficient sequence of $h(x)$. Recall that  a sequence $(a_i)_{i\geq 0}$ is log-concave if
\[
a_{i}^2\ge a_{i-1}a_{i+1}, \quad \mbox{ for } i\geq 1,
\]
and it has no internal zeros if $a_k\neq 0$ for  all $k: i<k<j$ whenever $a_i a_j \neq 0$. We have the following result.

\begin{cor} \label{thm:deg_h<2r_log-concave}
Fixing a positive integer $r$, let $h(x)$ be a polynomial with nonnegative coefficients of degree less than $2r$, say
$$h(x) = h_0 + h_1 x +h_2 x^2 + \cdots + h_{2r-1}x^{2r-1}.$$
If the coefficient sequence $(h_0, h_1, \ldots, h_{2r-1})$ is log-concave and has no internal zeros, then the polynomial sequence $\left(U_{r,r-1}^{n}\,h(x),\ldots, U_{r,1}^{n}\,h(x), U_{r,0}^{n}\,h(x)\right)$ is interlacing.
\end{cor}

\begin{proof} Without loss of generality, we may assume that all $h_i$ are positive. By Corollary \ref{thm:deg_h<2r}, it suffices to show that
 $$\frac{h_j}{h_{r+j}}\ge \frac{h_i}{h_{r+i}}$$
 for any $0\le i<j\le r-1$.
 By the log-concavity of $(h_0, h_1, \ldots, h_{2r-1})$, for any $0\le k<\ell \le r-1$ we have
  \begin{align*}
  \frac{h_\ell}{h_{\ell+1}} \ge \frac{h_{\ell-1}}{h_{\ell}} \ge \cdots  \ge \frac{h_{k+1}}{h_{k+2}} \ge \frac{h_{k}}{h_{k+1}},
 \end{align*}
 and thus $$\frac{h_\ell}{h_{k}} \ge \frac{h_{\ell+1}}{h_{k+1}}.$$
Hence, it follows that
  \begin{align*}
  \frac{h_j}{h_{i}} \ge \frac{h_{j+1}}{h_{i+1}} \ge \cdots  \ge \frac{h_{r+j}}{h_{r+i}}.
 \end{align*}
This completes the proof.
\end{proof}

The well-known Newton inequality says that the coefficients of a real-rooted polynomial must be log-concave. Thus it is possible to replace the log-concavity condition by the real-rootedness of
$h(x)$ in Corollary \ref{thm:deg_h<2r_log-concave}. In fact, if $h(x)$ has only real zeros, then we no longer need the condition
on the restriction of $\deg h(x)$. As the third application of Theorem \ref{thm:main-hnb}, we have the following result.

\begin{cor}\label{thm:real-zeros}
Suppose that $r$ is a positive integer and $h(x)$ is a polynomial with nonnegative coefficients. If $h(x)$ has only negative real zeros, then the polynomial sequence  $\left(U_{r,r-1}^{n}\,h(x),\ldots, U_{r,1}^{n}\,h(x), U_{r,0}^{n}\,h(x)\right)$ is interlacing.
\end{cor}

\begin{proof}
It is known that if $h(x)$ is a polynomial with nonnegative coefficients and it has only negative real zeros, then the polynomial sequence $( h^{\langle r,r-1 \rangle}(x)$, $\ldots, h^{\langle r,1 \rangle}(x), h^{\langle r, 0 \rangle}(x) )$ is interlacing, see {\cite[Theorem 7.65]{Fisk2006Polynomials}}.
By Theorem \ref{thm:main-hnb}, we obtain the desired interlacing of the polynomial sequence $\left(U_{r,r-1}^{n}\,h(x),\ldots, U_{r,1}^{n}\,h(x), U_{r,0}^{n}\,h(x)\right)$.
\end{proof}

The remaining of this section is mainly concerned with
the application of Corollary \ref{thm:real-zeros} to the real-rootedness of some combinatorially defined polynomials related to Eulerian polynomials. To be self-contained, we first give an overview of these polynomials. The Eulerian polynomials are not only of interest in combinatorics, but also
of significance in geometry \cite{Bjoerner1984Some, Dolgachev1994character, Stanley1980number, Stembridge1992Eulerian}. Recall that the Eulerian polynomials $A_{n}(x)$ appear as $h$-polynomials in the following rational power series:
\begin{align}\label{eq:A}
 \sum_{i=0}^{\infty} (i+1)^{n}x^i = \frac{A_{n}(x)}{(1-x)^{n+1}}.
\end{align}
Let $\mathfrak{S}_n$ denote the set of permutations of $[n]=\{1,2,\ldots,n\}$. For a permutation  $\pi=\left( \pi(1),\pi(2), \ldots, \pi(n) \right) \in \mathfrak{S}_n$, the descent number of $\pi$, denoted $\des\,(\pi)$, is  the number of  $i$ ($1\le i\le n-1$) such that $\pi(i)>\pi(i+1)$. It is well known that the Eulerian polynomial $A_n(x)$ can be treated as the descent generating function of $\mathfrak{S}_n$, namely,
$$A_{n}(x) = \sum_{\pi \in \mathfrak{S}_n} x^{\des\,(\pi)}.$$
A remarkable property of these polynomials is that they have only real zeros.

There are various generalizations of Eulerian polynomials. Here we focus on two families of polynomials related to $A_{n}(x)$, one of which is the following refinement
\begin{align}\label{eq:A-ell}
A_n^{\langle \ell \rangle}(x) = \sum_{\pi  \in \mathfrak{S}_n \atop \pi(1)=\ell} x^{\des\, (\pi)},
\end{align}
and the other is the descent generating function over colored permutations.  The polynomials $A_n^{\langle \ell \rangle}(x)$
have been shown to be real-rooted by Brenti \cite{Brenti2008f}.
The notion of descents was generalized to colored permutations by Steingr{\'{\i}}msson  \cite{Steingrimsson1992Permutations,Steingrimsson1994Permutation}.
Let $\ZZ_r$ be the cyclic group of order $r$.
For a positive integer $r$, the wreath product $\ZZ_r \wr \mathfrak{S}_n$, can be considered as  the set of   $r$-colored permutations,
written as
$
\pi = (\pi(1)_{\xi_1}, \pi(2)_{\xi_2}, \ldots, \pi(n)_{\xi_n}),
$
 where $(\pi(1), \ldots, \pi(n) ) \in \mathfrak{S}_n$ and
$(\xi_1, \ldots, \xi_n) \in \{0,1,\ldots,r-1\}^n$.
The descent number of $\pi \in \mathbb{Z}_r \wr {\mathfrak S_n}$ is defined as
\begin{align*}
\des\, (\pi)  \ = \ |\{i \in [n] \mid \xi_i > \xi_{i+1} \  {\rm or} \  \xi_i = \xi_{i+1} \ {\rm and} \ \pi(i) > \pi(i+1)\}|,
\end{align*}
with the convention that $\pi(n+1)=n+1,\ \xi_{n+1}=0$.
This means that the colored letters are totally ordered as
$$1_0<\cdots<n_0<1_1<\cdots<n_1<\cdots<1_{r-1}<\cdots<n_{r-1}.$$
The Eulerian polynomial of $\mathbb{Z}_r \wr {\mathfrak S}_n$ is defined to be the descent generating polynomial of $\mathbb{Z}_r \wr {\mathfrak S}_n$, namely,
\begin{equation*}
G_{n,r}(x) =  \ \sum_{\pi \in \mathbb{Z}_r \wr {\mathfrak S}_n} x^{\des\, (\pi)}.
\end{equation*}
Analogous to $A_n^{\langle \ell \rangle}(x)$, one can introduce the following refinements of $G_{n,r}(x)$:
\begin{align*}
G^{\ell, -}_{n,r}(x)&=\sum_{\pi \in \ZZ_r \wr \mathfrak{S}_n \atop {\pi(1) = \ell}}x^{\des\, (\pi)},\\
G^{-, c}_{n,r}(x) &= \sum_{\pi \in \ZZ_r \wr \mathfrak{S}_n \atop \xi_1=c} x^{\des\, (\pi)},\\
G^{\ell, {c}}_{n,r}(x) &= \sum_{\pi \in \ZZ_r \wr \mathfrak{S}_n \atop {\pi(1) = \ell,  \xi_1= c}}x^{\des\, (\pi)},
\end{align*}
where $1\le \ell \le n$ and $0\le c \le r-1$ are two integers.
The real-rootedness of these refined polynomials has been studied by Savage and Visontai \cite{Savage2015s} from the view point of $\mathbf{s}$-inversion sequences. We shall show that their results can also be derived from Corollary \ref{thm:real-zeros}.
Precisely, we shall prove the following result.

\begin{cor}[{\cite[Section~3.3]{Savage2015s}}] \label{cor-refined-color}
Let $G_{n,r}(x), G^{\ell, -}_{n,r}(x), G^{-, c}_{n,r}(x)$ and $G^{\ell, {c}}_{n,r}(x)$ be defined as above. Then we have the following result.
\begin{itemize}
\item[(i)] The polynomial sequence
 $\left( G^{-,0}_{n,r}(x), G^{-,1}_{n,r}(x), \ldots, G^{-,r-1}_{n,r}(x)\right)$
 is interlacing.

\item[(ii)] The polynomial $G_{n,r}(x)$ has only real zeros.

\item[(iii)] The polynomial sequence
 $\left( G^{\ell, {0}}_{n,r}(x), G^{\ell, {1}}_{n,r}(x), \ldots, G^{\ell, {r-1}}_{n,r}(x)\right)$
 is interlacing.

\item[(iv)] The polynomial $G^{\ell, -}_{n,r}(x)$ has only real zeros.
\end{itemize}
\end{cor}

We would like to point out that (ii) of Corollary \ref{cor-refined-color} was originally due to Steimgr\'{i}msson \cite[Theorem~3.19]{Steingrimsson1992Permutations}, and (i), (iii) and (iv) were implicitly stated by Savage and Visontai \cite{Savage2015s}.

In order to prove Corollary \ref{cor-refined-color} via Corollary \ref{thm:real-zeros}, we need to find suitable real-rooted polynomials $h(x)$ such that the polynomials $U_{r,i}^{n}\,h(x)$ correspond to  $G^{-,i}_{n,r}(x)$ or $G^{\ell,i}_{n,r}(x)$, respectively. It turns out that $A_{n}(x)$
and $A_{n}^{\langle \ell\rangle}(x)$ fulfill our purpose as shown below.

\begin{thm}\label{thm:c}
 For any positive integers $n$,  $r$ and $1 \le \ell \le n$,
 there holds
 \begin{align}
  (1 + x +  \cdots + x^{r-1})^n \, A_n (x) &=  G^{-,0}_{n,r}(x^r) + x \cdot \frac{G^{-,r-1}_{n,r}(x^r)}{x^r} + \cdots  + x^{r-1} \cdot \frac{G^{-,1}_{n,r}(x^r)}{x^r}, \label{eq:c}\\[6pt]
(1+x+\cdots+x^{r-1})^n A_{n}^{\langle \ell\rangle}(x)&
= G^{\ell, 0}_{n,r}(x^r) + x \cdot \frac{G^{\ell, {r-1}}_{n,r}(x^r)}{x^r} + \cdots + x^{r-1} \cdot \frac{G^{\ell, {1}}_{n,r}(x^r)}{x^r}.\label{eq:lc}
  \end{align}
\end{thm}

In order not  to disrupt the exposition of the applications of our main result, we defer the proof of Theorem \ref{thm:c} to Section
\ref{sect-Carlitz}.
Note that the left hand side of \eqref{eq:c} is the generating function of colored permutations of flag descent statistic, see \cite{Athanasiadis2014Edgewise, Bagno2007Colored}.
The $r=2$ case of \eqref{eq:c} was discussed in terms of half Eulerian polynomials, see \cite{Yang2015Real}.

We proceed to prove Corollary \ref{cor-refined-color}.

\begin{proof}[Proof of Corollary \ref{cor-refined-color}]
 We first prove (i). It is easy to check that the constant term of the polynomial $G^{-,i}_{n,r}(x)$ vanishes for each $1\leq i\leq r-1$. In view of the identity \eqref{eq:c} and the real-rootedness of $A_n(x)$, from Corollary \ref{thm:real-zeros} we derive the interlacing of the polynomial sequence $$\left( \frac{G^{-,1}_{n,r}(x)}{x}, \ldots, \frac{G^{-,r-1}_{n,r}(x)}{x}, G^{-,0}_{n,r}(x) \right).$$
 Multiplying each polynomial of the above sequence by $x$,  we get that the sequence $$\left( G^{-,1}_{n,r}(x), \ldots, G^{-,r-1}_{n,r}(x), x G^{-,0}_{n,r}(x) \right)$$ is interlacing.
 Since  $G^{-,i}_{n,r}(x) \sep x G^{-,0}_{n,r}(x)$ for all $1\le i \le r-1$, we have $G^{-,0}_{n,r}(x) \sep G^{-,i}_{n,r}(x)$. Therefore, the sequence $$\left( G^{-,0}_{n,r}(x), G^{-,1}_{n,r}(x), \ldots, G^{-,r-1}_{n,r}(x)\right)$$ is interlacing.
This completes the proof of (i).

We continue to prove (ii).  Taking
  \begin{align*}
  \left(f_{1}(x), f_{2}(x), \ldots, f_{r}(x)\right)&=(1,1,\ldots,1),\\
  \left(g_{1}(x), g_{2}(x), \ldots, g_{r}(x)\right)&=\left( G^{-,0}_{n,r}(x), G^{-,1}_{n,r}(x), \ldots, G^{-,r-1}_{n,r}(x)\right)
  \end{align*}
  in Lemma \ref{lem:fg_r},
  we obtain that the polynomial
  $$G^{\ell, -}_{n,r}(x) = \sum_{i=0}^{r-1} G^{\ell, i}_{n,r}(x)$$
   has only real zeros.

The proofs of (iii) and (iv) are exactly the same as those of (i) and (ii), which will be omitted here.
\end{proof}

\section{Proof of Theorem \ref{thm:c}}\label{sect-Carlitz}

The main aim of this section is to prove Theorem \ref{thm:c}.
To this end, we establish a Carlitz identity for refined  colored  permutations. Before presenting our identity, let us first review some related background.

A generalization of the classical Euler's identity \eqref{eq:A} was known before Carlitz \cite{Carlitz1975combinatorial}.
Let $[i]_q=1+q+\cdots+q^{i-1}$.
For any positive integer $n$, we define the $q$-factorial by $(x;q)_{n}=\prod_{i=0}^{n-1}(1-xq^i)$ with $(x;q)_{0}=1$.
Define the major index of $\pi$, denoted $\maj\, (\pi)$, to be the sum of the position $i$ ($1\le i\le n-1$) which contributes to its descent number, namely, $\pi(i)>\pi(i+1)$.
The following identity is known as Carlitz identity \cite{Carlitz1975combinatorial}:
\begin{align}\label{eq:A-q}
\sum_{i=0}^{\infty} [i+1]_q^n x^i = \frac{\sum_{\pi \in \mathfrak{S}_n} x^{\des\, (\pi)} q^{\maj\, (\pi)}}{(x;q)_{n+1}}.
\end{align}

An analogue of Carlitz identity for  colored permutations was given by Chow and Mansour  \cite{Chow2011Carlitz}.
The major index  of a colored permutation $
\pi = (\pi(1)_{\xi_1}, \pi(2)_{\xi_2}, \ldots, \pi(n)_{\xi_n})
$, denoted $\maj\, (\pi)$, is defined as the sum of descent positions $i$,
and the flag major index of  $\pi$ is defined as
$\fmaj\, (\pi) = r \cdot \maj\,(\pi)-\sum_{i=1}^{n}\xi_{i}.$
Chow and Mansour  \cite{Chow2011Carlitz} obtained the following identity:
\begin{align} \label{eq:CM}
\sum_{i=0}^{\infty} [r i+1]_q^n x^i = \frac{\sum_{\pi \in \mathbb{Z}_r \wr {\mathfrak S}_n} x^{\des\, (\pi)} q^{\fmaj\, (\pi)}}{(x;q^r)_{n+1}}.
\end{align}
In particular, the $r=2$ case of \eqref{eq:CM} was already known to Chow and Gessel \cite{Chow2007descent}.

For $1\le \ell \le n$ and $0\le c \le r-1$, let
\begin{align*}
 G^{\ell, {c}}_{n,r}(x,q) = \sum_{\pi \in \ZZ_r \wr \mathfrak{S}_n \atop {\pi(1) = \ell,  \xi_1= c}} x^{\des\, (\pi)}q^{\fmaj\, (\pi)}.
\end{align*}
Clearly, $ G^{\ell, {c}}_{n,r}(x) =   G^{\ell, {c}}_{n,r}(x,1)$.
The first main result of this section is the following refinement of \eqref{eq:CM}.

\begin{thm}\label{thm:Carlitz}
  For $1\le \ell \le n$ and $0\le c \le r-1$, there holds
\begin{equation}\label{eq:refCarlitz}
  \frac{ G^{\ell, {c}}_{n,r}(x,q)}{(xq^r;q^r)_{n}} = \delta_{\ell,1}\delta_{c,0} + \sum_{i=1}^{\infty} \left( [r i-c]_{q}^{\ell-1} [r i-c+1]_{q}^{n-\ell}  q^{ri-c}\right)  x^i.
\end{equation}
\end{thm}


Our proof of Theorem \ref{thm:Carlitz} adopts the method of ``balls in boxes''. For more information about this method, see \cite{Lin2015descent, Moynihan2012Flag, Petersen2013Two}.
Before giving the proof, let us first introduce some notations and definitions. Given $\pi \ = \
(\pi(1)_{\xi_1}, \pi(2)_{\xi_2}, \ldots, \pi(n)_{\xi_n}) \in \ZZ_r \wr \mathfrak{S}_n$, there are $n+1$ positions between the letters of $(\pi)$ and on the ends. These positions are labelled $0,1,\ldots,n$ from left to right with position $i$
between $\pi(i)_{\xi_{i}}$  and $\pi(i+1)_{\xi_{i+1}}$.
Let $N_{j}(\pi)$ be the number of letters of $\pi$ with color $j$.
Define a barred colored permutation
to be a shuffle of a colored permutation with a sequence of bars
such that between any two bars the letters are increasing
and such that only letters of color $0$ are allowed in the rightmost compartment. For $1\le \ell \le n$ and $0\le c \le r-1$, let $\mathcal{M}_{\ell,c}$ denote the set of all barred colored permutations beginning with $\ell_c$, which implies that there exist no bars at position $0$.

\begin{proof}[Proof of Theorem \ref{thm:Carlitz}]
We shall prove the theorem by counting weighted barred colored permutations of $\mathcal{M}_{\ell,c}$
in two different ways. 

Beginning with a colored permutation $\pi \in \ZZ_r \wr \mathfrak{S}_n$ with $\pi(1) = \ell,  \xi_1= c$, we must  place a bar in each descent position.
We weight each bar in position $i$ by $x q^{i}$ and each letter with color $j$ by $z_{j}$. Define the weight of a barred colored permutation to be the product of weights of bars  and  letters. These bars placed in descent positions contribute a factor of $x^{\des\, (\pi)}q^{\maj\, (\pi)}$.
Next, we place arbitrary many bars in any positions but no bars in position $0$, and the sum of the products of their possible weights is
$$\prod_{i=1}^{n} \left( 1+xq^i+(xq^i)^2+\cdots \right)
= \frac{1}{(xq;q)_{n}}.$$
Hence $\pi$ totally contributes
 $$\frac{x^{\des\, (\pi)}q^{\maj\, (\pi)}  \prod_{j=0}^{r-1}z_{j}^{N_{j}(\pi)}}{(xq;q)_{n}}$$
 to the sum of the weights of all barred colored permutations of $\mathcal{M}_{\ell,c}$. Summing over all $\pi\in \ZZ_r \wr \mathfrak{S}_n$ with $\pi(1) = \ell,  \xi_1= c$, we have
 \begin{align*}
  \frac{ \sum_{\pi \in \ZZ_r \wr \mathfrak{S}_n \atop {\pi(1) = \ell,  \xi_1= c}} x^{\des\, (\pi)}q^{\maj\, (\pi)} \prod_{j=0}^{r-1}z_{j}^{N_{j}(\pi)}}{(xq;q)_{n}}.
 \end{align*}
If we substitute $q$ by $q^{r}$ and $z_{j}$
 by $q^{-j}$ for $0\le j\le r-1$, then the above expression turns out to be
 \begin{align*}
  \frac{ \sum_{\pi \in \ZZ_r \wr \mathfrak{S}_n \atop {\pi(1) = \ell,  \xi_1= c}} x^{\des\, (\pi)}q^{\fmaj\, (\pi)}}{(xq^r;q^r)_{n}}=\frac{ G^{\ell, {c}}_{n,r}(x,q)}{(xq^r;q^r)_{n}}.
 \end{align*}

On the other hand,  we count weighted barred colored permutations of $\mathcal{M}_{\ell,c}$ but instead start with $i$ bars after $ \ell_{c}$. With $i$ initial bars given, there are $i+1$ compartments with the leftmost compartment beginning with  $\ell_{c}$.
We label the compartments $0,1,\ldots, i$ from right to left. Now we weight each bar by $x$ and each letter with color $j$ in compartment $k$ by $q^{k}z_{j}$. It is straightforward to verify that the weight of a barred colored permutation is still the product of weights of bars and letters of $\pi$, namely, these two different weighting methods agree for every barred colored permutation. 

Let $\Omega(i,q,z_0,\ldots,z_{r-1})$ denote the sum of the products of the weights of the letters over all barred colored permutations of $\mathcal{M}_{\ell,c}$ with initial $i$ bars.
Note that, with $i$ bars given, to create a barred colored permutation in $\mathcal{M}_{\ell,c}$, we must choose a color and a compartment for each element $j\in [n] \backslash \ell$ (for which we choose color $c$ and compartment $i$) such that
$\ell_c$ is the smallest element in compartment $i$. Given such a choice, there is a unique way to place all the colored letters in the same compartment in increasing order.
This implies that the necessary choices for each $j\neq \ell$ are independent subject to certain conditions, and so $\Omega(i,q,z_0,\ldots,z_{r-1})$ factors
as
$$\Omega(i,q,z_0,\ldots,z_{r-1})=\prod_{j=1}^n \Omega_j(i,q,z_0,\ldots,z_{r-1}),$$
where $\Omega_j(i,q,z_0,\ldots,z_{r-1})$ denotes the weight contributed by all possible choices of the letter $j$.

Since the color of $\ell$ is fixed to be $c$ and the compartment of $\ell$ is fixed to be $i$, we have $\Omega_\ell(i,q,z_0,\ldots,z_{r-1})=q^iz_c$. To compute
$\Omega_j(i,q,z_0,\ldots,z_{r-1})$ for $j\neq \ell$, we consider separately the cases $c=0$ and $c\neq 0$ in view of the requirement that the rightmost compartment contain only letters colored  $0$.

First we consider the case of $c=0$. If $j<\ell$, we consider its contribution according to its color. If $j$ is colored $0$, then it can be placed in any compartment but compartment $i$, hence it contributes $(1+q+\cdots+q^{i-1})z_0=[i]_{q}z_{0}$ to $\Omega_j(i,q,z_0,\ldots,z_{r-1})$. If $j$ is colored $k$ ($1\le k \le r-1$), then it can be placed in any compartment but compartment $0$ (keeping in mind that only letters of color $0$ are allowed in compartment $0$) and hence it contributes
$(q+q^2+\cdots+q^i)z_k=q [i]_{q}z_{k}$ to 
$\Omega_j(i,q,z_0,\ldots,z_{r-1})$. 
Therefore, we have 
$$\Omega_j(i,q,z_0,\ldots,z_{r-1})=[i]_{q}z_{0} + q [i]_{q}\sum_{k=1}^{r-1}z_{k}.$$
Similarly, for any $j>\ell$
 $$\Omega_j(i,q,z_0,\ldots,z_{r-1})=[i+1]_{q}z_{0} + q [i]_{q}\sum_{k=1}^{r-1}z_{k}.$$
Therefore, 
\begin{align}
\Omega(i,q,z_0,\ldots,z_{r-1})= & q^i z_c \left([i]_{q}z_{0} + q [i]_{q}\sum_{k=1}^{r-1}z_{k}\right)^{\ell-1}
 \left([i+1]_{q}z_{0} + q [i]_{q}\sum_{k=1}^{r-1}z_{k}\right)^{n-\ell}.\label{eq-omega1}
\end{align}

We proceed to consider the case of $c>0$. For $j<\ell$, we consider its contribution according to its color. If $j$ is colored $0$, then it can be placed in any compartment but compartment $i$, hence it contributes $(1+q+\cdots+q^{i-1})z_0=[i]_{q}z_{0}$ to $\Omega_j(i,q,z_0,\ldots,z_{r-1})$. 
If $j$ is colored $k$ with $1\le k \le c-1$, then it can be placed in any compartment except for compartment $0$ and compartment $i$,  and hence it contributes
$(q+q^2+\cdots+q^{i-1})z_k=q [i-1]_{q}z_{k}$ to
$\Omega_j(i,q,z_0,\ldots,z_{r-1})$. 
If $j$ is colored $k$ with $k>c$, then it can be placed in any compartments except for compartment $0$,  and hence it contributes
$(q+q^2+\cdots+q^{i})z_k=q [i]_{q}z_{k}$ to
$\Omega_j(i,q,z_0,\ldots,z_{r-1})$. Therefore, for $j<\ell$ we have
$$\Omega_j(i,q,z_0,\ldots,z_{r-1})=[i]_{q}z_{0}+q[i-1]_{q}\sum_{k=1}^{c-1}z_{k}+q[i]_{q}\sum_{k=c}^{r-1}z_{k}.$$
Similarly, for any $j>\ell$ we have
 $$\Omega_j(i,q,z_0,\ldots,z_{r-1})=[i]_{q}z_{0}+q[i-1]_{q}\sum_{k=1}^{c}z_{k}+q[i]_{q}\sum_{k=c+1}^{r-1}z_{k}.$$
To summarize, we get
\begin{align}
\Omega(i,q,z_0,\ldots,z_{r-1})&= q^i z_c \left([i]_{q}z_{0}+q[i-1]_{q}\sum_{k=1}^{c-1}z_{k}+q[i]_{q}\sum_{k=c}^{r-1}z_{k}\right)^{\ell-1}\nonumber\\
&\quad \times   \left([i]_{q}z_{0}+q[i-1]_{q}\sum_{k=1}^{c}z_{k}+q[i]_{q}\sum_{k=c+1}^{r-1}z_{k}\right)^{n-\ell} .\label{eq-omega2}
\end{align}
Substituting $q$ by $q^{r}$ and $z_{k}$
 by $q^{-k}$ for $0\le k\le r-1$, it is tedious to verify that both \eqref{eq-omega1} and
 \eqref{eq-omega2} become
 \begin{align*}
 \Omega(i,q,z_0,\ldots,z_{r-1})|_{{{q \rightarrow q^r}\atop{z_k \rightarrow q^{-k}}}} = q^{ri-c} [r i-c]_{q}^{\ell-1}  [r i-c+1]_{q}^{n-\ell}.
 \end{align*}
Then multiplying by $x^i$ and summing over $i$, it follows that 
 \begin{equation*}
  \frac{ G^{\ell, {c}}_{n,r}(x,q)}{(xq^r;q^r)_{n}} = \delta_{\ell,1}\delta_{c,0} + \sum_{i=1}^{\infty} \left( [r i-c]_{q}^{\ell-1} [r i-c+1]_{q}^{n-\ell}  q^{ri-c}\right)  x^i,
\end{equation*}
as desired. 
This completes the proof.
 \end{proof}


We proceed to prove Theorem \ref{thm:c}.
\begin{proof}[Proof of Theorem \ref{thm:c}]
We first prove the identity \eqref{eq:c}.
By letting $q$ be 1,  \eqref{eq:refCarlitz} becomes
\begin{align}\label{eq:lc-key}
 \frac{G^{\ell, {c}}_{n,r}(x)}{(1-x)^n} = \delta_{\ell,1}\delta_{c,0}+\sum_{i=1}^{\infty}(r i-c)^{\ell-1}(r i-c+1)^{n-\ell} x^i.
\end{align}
Summing over $\ell$ from $1$ to $n$ in \eqref{eq:lc-key}, we get
\begin{align*}
\frac{G^{-, c}_{n,r}(x)}{(1-x)^n} & = \frac{\sum_{\ell=1}^{n}G^{\ell, {c}}_{n,r}(x)}{(1-x)^n}=
\delta_{c,0} +  \sum_{i=1}^{\infty} \sum_{\ell=1}^{n} (r i-c)^{\ell-1}(r i-c+1)^{n-\ell}  x^i.
\end{align*}
By the geometric summation formula, we have
\begin{align}\label{eq-temp}
\frac{G^{-, c}_{n,r}(x)}{(1-x)^n} & =  \delta_{c,0} +  \sum_{i=1}^{\infty} \left( (ri-c+1)^{n}-(ri-c)^{n}\right) x^i.
\end{align}
On the other hand, by \eqref{eq:A} we obtain
\begin{align*}
 \frac{A_{n}(x)}{(1-x)^{n}} = (1-x) \sum_{j=0}^{\infty} (j+1)^{n}x^j
  =  \sum_{j=0}^{\infty} ((j+1)^{n}-j^n)x^j.
\end{align*}
Decomposing the above summation into $r$ parts by the residue classes of $j$ modulo $r$, it follows that
\begin{align*}
\frac{A_{n}(x)}{(1-x)^{n}}  & = 1 +  \sum_{c=0}^{r-1} \sum_{i=1}^{\infty} \left( (ri-c+1)^{n}-(ri-c)^{n}\right) x^{ri-c}.
\end{align*}
Then by \eqref{eq-temp} we get
\begin{align*}
\frac{A_{n}(x)}{(1-x)^{n}}  & = \sum_{c=0}^{r-1} \frac{G^{-, c}_{n,r}(x^r)}{x^c\, (1-x^r)^n}.
\end{align*}
Multiplying both sides of the above formula by $(1-x^r)^n$ leads to
\begin{align*}
(1+x+\cdots+x^{r-1})^n A_{n}(x) =  \frac{G^{-, {r-1}}_{n,r}(x^r)}{x^{r-1}} + \cdots +  \frac{G^{-, {1}}_{n,r}(x^r)}{x} + G^{-, 0}_{n,r}(x^r),
\end{align*}
as desired.
This completes the proof of  \eqref{eq:c}.

We proceed to the proof of \eqref{eq:lc}, in the same way as that of \eqref{eq:c}.
For $0\le i\le r-1$, we have
$A_{n}^{\langle \ell\rangle}(x) = G^{\ell, 0}_{n,1}(x)$. Hence by \eqref{eq:lc-key} we get
\begin{align*}
 \frac{A_n^{\langle \ell\rangle}(x)}{(1-x)^n} = \delta_{\ell,1} + \sum_{j=1}^{\infty} j^{\ell-1}(j+1)^{n-\ell}x^j.
\end{align*}
Decomposing the summation on the right hand side into $r$ parts by the residue classes of $j$ modulo $r$, we get
\begin{align*}
 \frac{A_n^{\langle \ell\rangle}(x)}{(1-x)^n}= \delta_{\ell,1} +
 \sum_{c=0}^{r-1}\sum_{i=1}^{\infty} (ri-c)^{\ell-1}(ri-c+1)^{n-\ell}x^{ri-c}
\end{align*}
Then by \eqref{eq:lc-key} we have
\begin{align*}
\frac{A_n^{\langle \ell\rangle}(x)}{(1-x)^n}  & = \sum_{c=0}^{r-1} \frac{G^{\ell, c}_{n,r}(x^r)}{x^c\, (1-x^r)^n}.
\end{align*}
Multiplying the both sides of the above formula by $(1-x^r)^n$ leads to
\begin{align*}
(1+x+\cdots+x^{r-1})^n A_{n}^{\langle \ell\rangle}(x)
  & = \frac{G^{\ell, {r-1}}_{n,r}(x^r)}{x^{r-1}} + \cdots + \frac{G^{\ell, {1}}_{n,r}(x^r)}{x}+ G^{\ell, 0}_{n,r}(x^r)
  \end{align*}
as desired. This completes the proof.
\end{proof}

\vskip 3mm
\noindent {\bf Acknowledgments.}
We would like to thank the referee for the valuable comments which helped to improve the manuscript.
This work was partially supported by the NSFC  (grant 11626172, 11701424)  and the PHD Program 52XB1616 of Tianjin Normal University.


\end{document}